\documentclass{amsart}
\usepackage{amsmath, amsthm, amssymb, amscd, mathrsfs, eucal, epsfig, multirow, longtable, appendix}
\usepackage{pinlabel}

\DeclareMathOperator{\GL}{\mathrm{GL}}

\input xy
\xyoption{all}

\usepackage{hyperref}

\begin{document}

\newtheorem{theorem}{Theorem}
\newtheorem{lemma}[theorem]{Lemma}
\newtheorem{corollary}[theorem]{Corollary}
\newtheorem{conjecture}[theorem]{Conjecture}
\newtheorem{proposition}[theorem]{Proposition}
\newtheorem{question}[theorem]{Question}
\newtheorem{problem}[theorem]{Problem}
\newtheorem*{PL_LO_thm}{PL is LO~\ref{thm:PL_LO}}
\newtheorem*{claim}{Claim}
\newtheorem*{criterion}{Criterion}
\theoremstyle{definition}
\newtheorem{definition}[theorem]{Definition}
\newtheorem{construction}[theorem]{Construction}
\newtheorem{notation}[theorem]{Notation}
\newtheorem{convention}[theorem]{Convention}
\newtheorem*{warning}{Warning}

\theoremstyle{remark}
\newtheorem{remark}[theorem]{Remark}
\newtheorem{example}{Example}
\newtheorem*{case}{Case}

\def\C{\mathbb C}
\def\F{\mathfrak F}
\def\H{\mathbb H}
\def\R{\mathbb R}
\def\Z{\mathbb Z}
\def\RP{\mathbb{RP}}
\def\PL{\textnormal{PL}}
\def\GL{\textnormal{GL}}
\def\Homeo{\textnormal{Homeo}}
\def\a{{\alpha}}
\def\b{{\beta}}
\def\s{{\sigma}}
\def\G{{\Gamma}}
\def\Homeo{\textnormal{Homeo}}
\def\Diff{\textnormal{Diff}}
\def\fix{\textnormal{fix}}
\def\fro{\textnormal{fro}}
\def\Id{\textnormal{Id}}

\title{Generalized torsion in knot groups}

\author{Geoff Naylor}
\author{Dale Rolfsen}

\date{\today}
\begin{abstract}  We show that for many classical knots one can find generalized torsion in the fundamental group of its complement, commonly called the knot group.  It follows that such a group is not bi-orderable.  Examples include all torus knots, the (hyperbolic) knot $5_2$ and satellites of these knots.

\end{abstract}
\maketitle

\section{Introduction}
The purpose of this note is to initiate a study of generalized torsion in classical knot groups.  For two elements $x$ and $y$ of a group, we use the notation $x^y := y^{-1}xy$ for the conjugate of $x$ by $y$ and $[x, y] := x^{-1}y^{-1}xy$ for their commutator.  The identity of a group is denoted by $1$.
A {\em generalized torsion} element of a group is an element $x \ne 1$ for which some (nonempty finite) product of conjugates of that element is the identity: $x^{y_1}x^{y_2}\cdots x^{y_k} = 1.$ 

\begin{example}\label{klein}
An example of generalized torsion is the fundamental group of the Klein bottle 
$\langle x, y | y^{-1}xy = x^{-1} \rangle$ in which $x$ is a generalized torsion element: $x^yx= 1.$ 
\end{example}

\section{Knot groups and bi-ordering}

A classical knot is a subset $K$ of euclidean space $\R^3$ which is abstractly homeomorphic to the circle $S^1$.  We assume $K$ is smooth or piecewise linear.
The {\em knot group} of $K$ is the fundamental group of its complement: $\pi_1(\R^3 \setminus K)$.  It has long been known that knot groups do not contain torsion, that is, elements of finite order.  In fact, knot groups have the stronger property of being locally indicable: every nontrivial finitely-generated subgroup surjects to $\Z$, the integers (see \cite{HS} and \cite{BRW}).  It follows that knot groups support left-invariant orderings, meaning that the elements of the group can be given a total ordering $<$ such that for elements $x,y, z$ of the group one has $y < z
\iff xy < xz$.  The groups of some knots, for example $4_1$, are bi-orderable: there is a strict total ordering of the elements which is invariant under multiplication on both sides (see \cite{PR}).  But not all knot groups are bi-orderable; generalized torsion is a well-known obstruction to bi-orderability.  The following is clear, for if we assume without loss that $x>1$ in a bi-ordering, then all its conjugates are also greater than the identity, so 
$x^{y_1}x^{y_2}\cdots x^{y_k}  > 1$.

\begin{proposition}\label{bi-ordNoGenTors}
Bi-orderable groups do not have generalized torsion elements.
\end{proposition}

Although left-orderable groups cannot have torsion, they can have generalized torsion -- the Klein bottle group of Example \ref{klein} is an instance: in fact it has exactly four left-orderings.  As we'll see, many knot groups are also examples.

Among the various reasons bi-orderability of a knot group is of interest is the following result of \cite{CR}.

\begin{theorem}
If the group of the knot $K$ is bi-orderable, then surgery on $K$ cannot produce a 3-manifold with finite fundamental group, or more generally any L-space in the sense of Ozsv{\'a}th and Szab{\'o} \cite{OS}.
\end{theorem}

\section{Commutators as generalized torsion}

\begin{example}\label{ex:trefoil}
Consider the group of the trefoil knot $G_{2,3} = \langle x, y : x^2 = y^3 \rangle$.  The commutator 
$[x ,y]$ is a nontrivial element, because $G$ is non-abelian.  But 
$x^{-1}[x, y]x[x ,y] = x^{-2}y^{-1}xyxx^{-1}y^{-1}xy = x^{-2}y^{-1}x^2y = y^{-3}y^{-1}y^3y = 1.$  So $[x ,y]$ is a generalized torsion element in the knot group of the trefoil.  This is expanded in Theorem \ref{torus group} below for other torus knots -- that is, knots that can be inscribed on the surface of an unknotted solid torus in $\R^3$.
\end{example}

The following lemma is a standard fact of group theory, easily proved by induction using the identities
$[x^n,y] = [x^{n-1},y]^x[x,y]$ and $[x,y^n] = [x,y][x,y^{n-1}]^y$.  Its relation to orderability is pointed out in  \cite{CGW}, Examples 2.1 and 2.2, which also inspired the remaining results of this section.

\begin{lemma}\label{lemma: commutator}
In any group, for all positive integers $p, q$ the commutator $[x^p, y^q]$ is a product of conjugates of 
$[x, y]$.
\end{lemma}

\begin{proposition}\label{proposition: commuting}  Suppose $G$ is any group containing elements $x$ and $y$ which do not commute, but for which some positive powers $x^p$ and $y^q$ do commute.  Then the commutator $[x, y]$ is a generalized torsion element of $G$.
\end{proposition}

\begin{proof}
By Lemma \ref{lemma: commutator}, $[x^p, y^q]$ is a product of conjugates of $[x, y]$.  But $[x, y] \ne 1$ and $[x^p, y^q] = 1$.
\end{proof}

\begin{theorem}\label{torus group}
For any integers $p$ and $q$ with $|p| > 1$ and $|q| > 1$, the group 
$$G = G_{p,q} = \langle x, y : x^p = y^q \rangle$$ 
contains a generalized torsion element, namely the commutator $[x, y].$
\end{theorem}

\begin{proof}
The conditions on $p$ and $q$ imply $G$ is nonabelian, so $[x, y] \ne 1$. Proposition 
\ref{proposition: commuting} completes the proof.
\end{proof}

If $p$ and $q$ are relatively prime integers, $G_{p,q}$ is the group of the $p,q$-torus knot (two examples pictured below).

\begin{corollary}
The group of a nontrivial torus knot has generalized torsion.
\end{corollary}

\section{The first few prime knots}

Let us consider the nontrivial prime knots up to six crossings, in their usual tabulated order:

\bigskip

\noindent \includegraphics[scale=0.2]{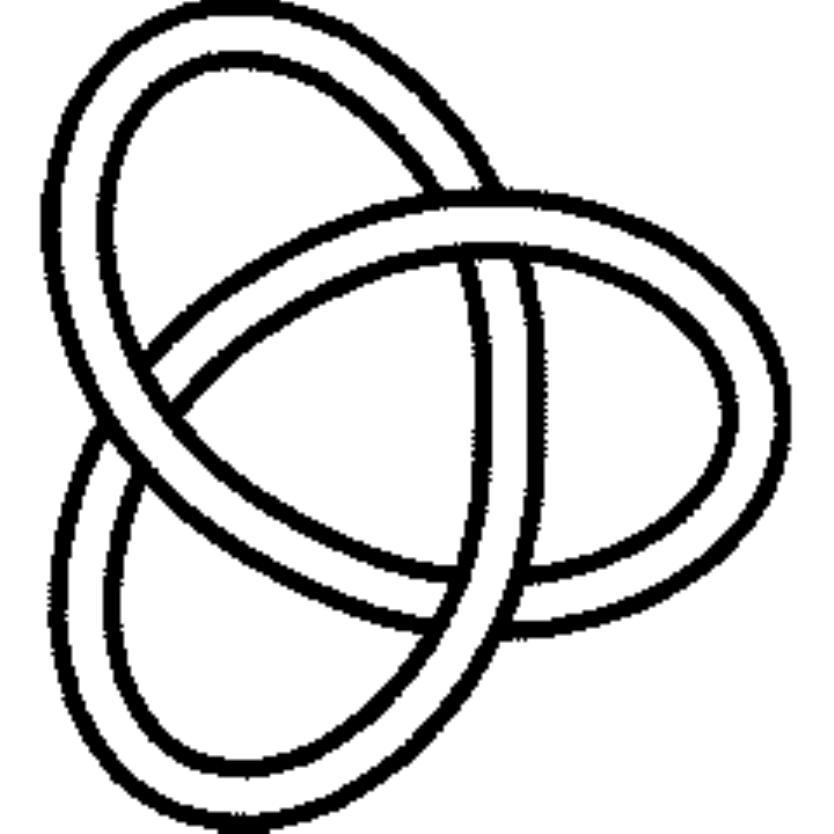}$3_1$:
 The trefoil, or $2, 3-$torus knot.  Its group has generalized torsion as observed in Example \ref{ex:trefoil}.

\bigskip

\noindent\includegraphics[scale=0.2]{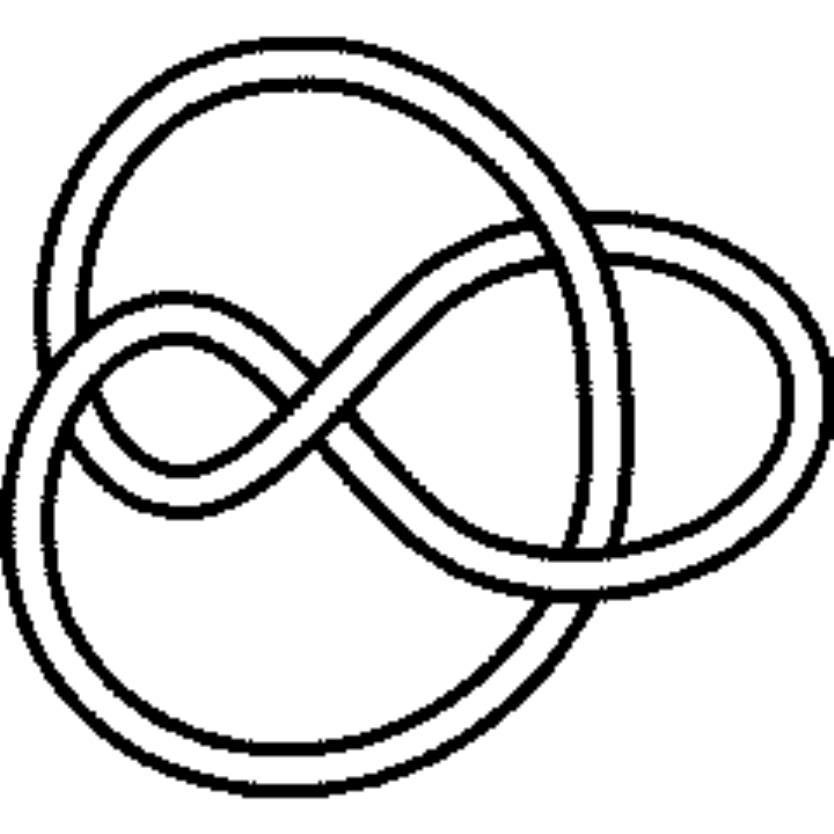}$4_1$:
The figure-eight knot.  It has group $\langle a, b | ab^3a = ba^2b \rangle$.  Its Alexander polynomial,
$1 - 3t + t^2$, has real positive roots.  It is shown in \cite{PR} that its group is bi-orderable using that fact, and that $4_1$ is a fibred knot (see below for definition).   Hence NO generalized torsion.

\bigskip
   
\noindent \includegraphics[scale=0.2]{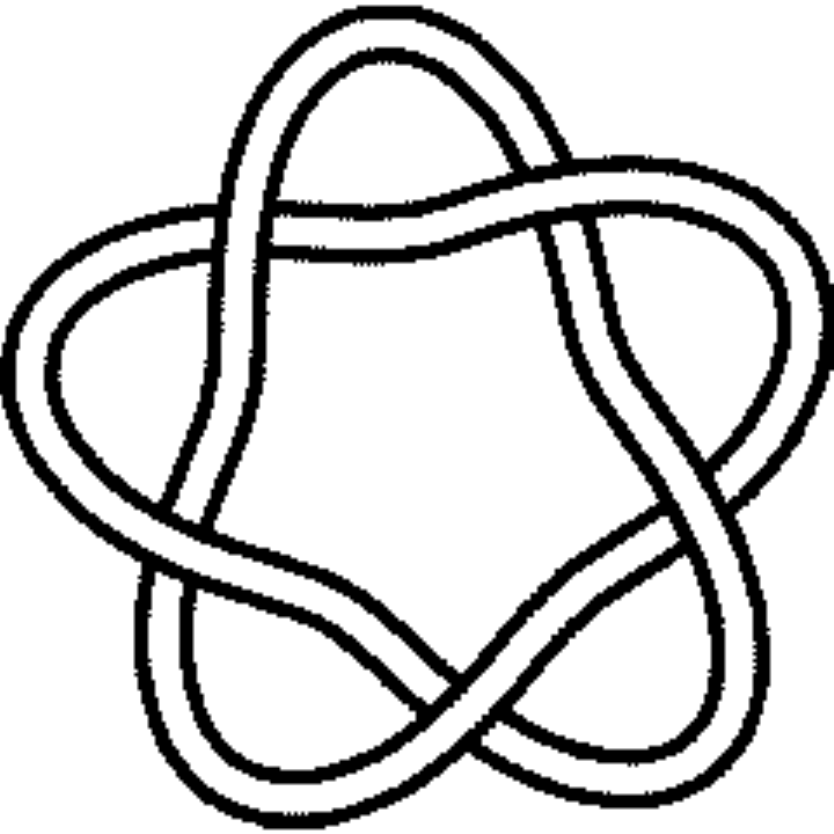}$5_1$:
The $2, 5-$torus knot or cinquefoil.  Generalized torsion here by Theorem \ref{torus group}.

\bigskip

All the above knots are fibred knots; that is, their complements in $S^3 = \R^3 + \infty$ fibre over $S^1$ with fibres being open surfaces whose closure has the knot as its boundary.

\bigskip

\noindent \includegraphics[scale=0.2]{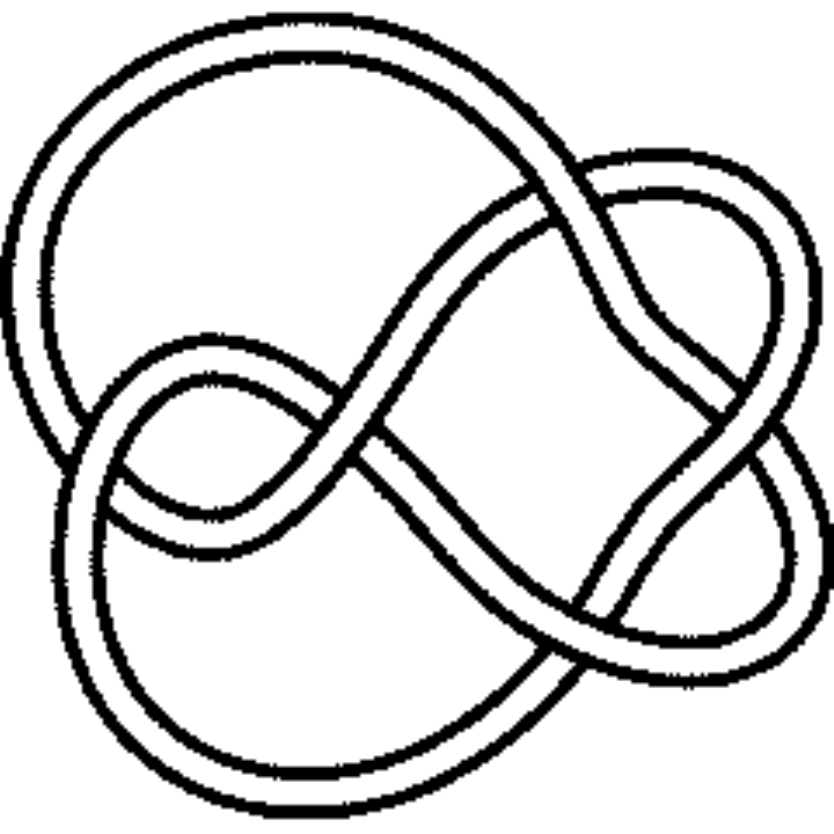}$5_2$:  We'll see this knot's group also has generalized torsion.

\bigskip

The knot $5_2$ is the first knot in the standard table which is not fibred.   It has Alexander polynomial $2t^2 -3t + 2$, whereas fibred knots have {\em monic} polynomials.  Its knot group $G_{5_2}$ has presentation
$$G_{5_2} = \langle a, b | b^2a^{-2}b^2 = a^{-1}b^3a^{-1} \rangle.$$   It is shown in \cite{CGW} that 
$G_{5_2}$ is not 
bi-orderable using the fact that the Alexander polynomial has no real roots.  In this note we strengthen this result by noting there is generalized torsion in $G_{5_2}$. 

\begin{theorem}\label{5_2}
In the group $G_{5_2}$, $a^{-1}bab^{-1}$ is a generalized torsion element. 
\end{theorem}

This will be proved in the next section.
 
\bigskip

The knot in the tables which follows $5_2$ is sometimes known as the ``stevedore knot'' ...

\noindent \includegraphics[scale=0.2]{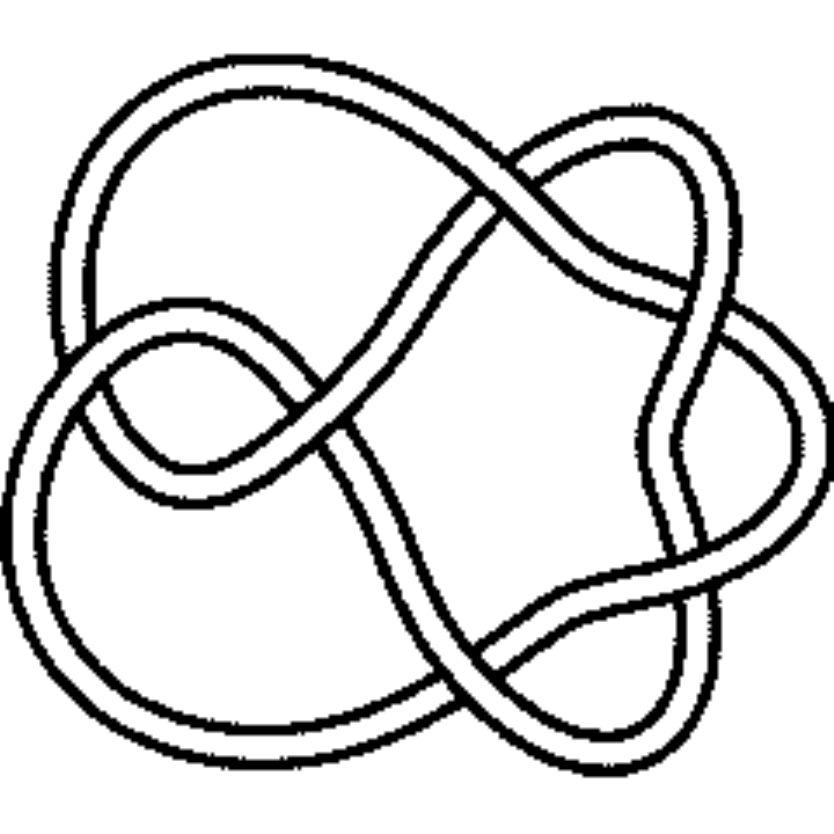}$6_1$:
This knot has Alexander polynomial $2 -5t + 2t^2$, whose roots are $1/2$ and 2.  According to \cite{CDN}, using results of  \cite{CGW}, the group of this knot is bi-orderable.  Therefore its group does NOT have generalized torsion.

The next knot in the tables is still a mystery ...

\noindent \includegraphics[scale=0.2]{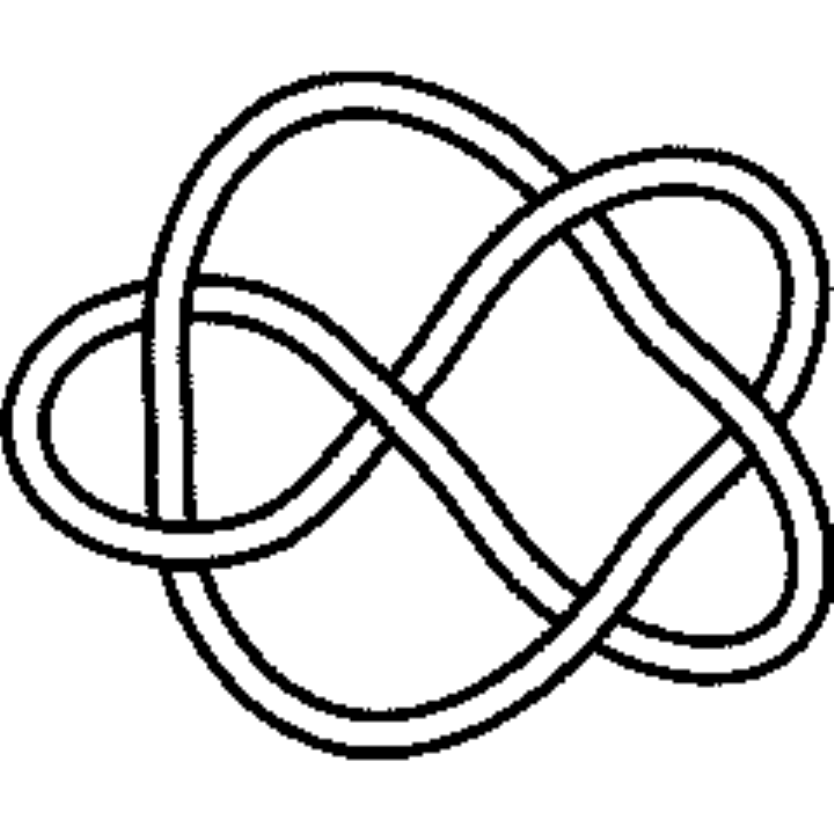}$6_2$:  This is the first knot in the tables for which bi-orderability, as well as the existence of generalized torsion, is unknown at this writing.  Its Alexander polynomial is
$1 -3t +3t^2 -3t^3 +t^4$, which has two positive real roots $2.15372\dots$ and $0.446431\dots$ and two complex roots, approximately $0.19098\pm 0.98159i$.  
 \bigskip
 
The last prime knot with at most six crossings is

\noindent \includegraphics[scale=0.2]{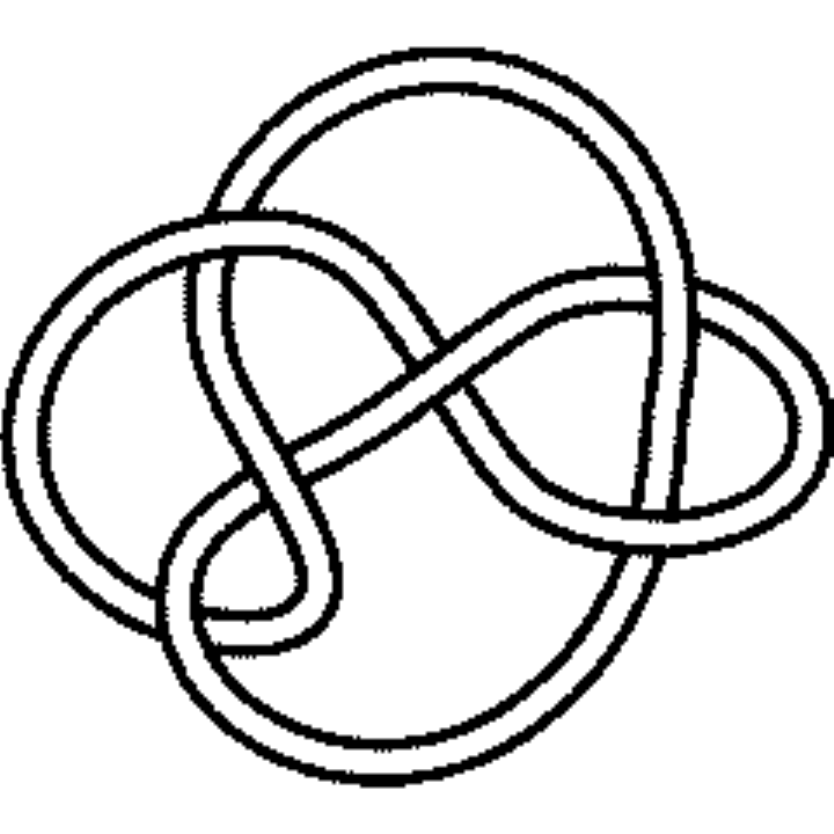}$6_3$: Its Alexander polynomial is $1 - 3t + 5t^2 - 3t^3 + t^4$ which has no real roots.  Using this fact, it is shown in \cite{CDN}, again using results of \cite{CGW}, that the group of this knot is not bi-orderable.  We do not know if its group contains generalized torsion.

\section{Proof of Theorem \ref{5_2}}

We adopt the convention of using capital letters to denote the inverse of a group element, so that the defining relation of the group $G_{5_2}$ becomes $b^2A^2b^2 = Ab^3A$.   The existence of generalized torsion in this group was discovered with the help of a Python program written by the first named author to test for bi-orderability.  The first proof we offer is adapted from the non-bi-orderability argument discovered by computer, and seems rather unmotivated.  We will also present a second proof, suggested by Andrew Glass,  which the reader may find somewhat more mathematically satisfying.

\begin{proof}
First we need to argue that $AbaB$ is not the identity.  But this follows since the group, like all groups of nontrivial knots, is nonabelian.

Let $X$ denote the set of elements of $G_{5_2}$ which are (nonempty) products of conjugates of 
$AbaB$.  Our goal is to show that $1 \in X$.  It is clear that $X$ is closed under multiplication and conjugation, and of course:
\begin{equation}
AbaB \in X
\end{equation}
Now $B^2(AbaB)b^2 = B^2Abab,$ so 
\begin{equation}
B^2Abab \in X
\end{equation}
Conjugating $AbaB$ by $B$ shows that
\begin{equation}
bAbaB^2 \in X
\end{equation}
Now consider the product of (1) and (3): $(AbaB)(bAbaB^2) = Ab^2aB^2$ which is equal to 
$aB^2Ab$ by the following calculation.  $(Ab^2aB^2)(aB^2Ab)^{-1} = (Ab^2aB^2)(Bab^2A)
= Ab^2(aB^3a)b^2A = Ab^2(B^2a^2B^2)b^2A = 1.$  Here we used the relation 
$aB^3a = B^2a^2B^2$, which follows from the defining relation.  So we have shown
\begin{equation}
aB^2Ab \in X
\end{equation}
Conjugating (4) gives $$(BaB^2)(aB^2Ab)(b^2Ab) = BaB^2aB^2(Ab^3A)b
= BaB^2aB^2(b^2A^2b^2)b = BaB^2Ab^3.$$  This last expression equals $BAb^2a$ from the calculation
$$(BaB^2Ab^3)(AB^2ab) =  BaB^2(Ab^3A)B^2ab =  BaB^2(b^2A^2b^2)B^2ab = 1,$$ so we conclude
\begin{equation}
BAb^2a \in X
\end{equation}
Another conjugate of (4) gives $b(aB^2Ab)B = baB^2A$ so
\begin{equation}
baB^2A \in X
\end{equation}
Next we conjugate (1): $ba(AbaB)AB = b^2aBAB$ so that 
\begin{equation}
b^2aBAB \in X
\end{equation}
Multiplying those last two elements of $X$ yields $(b^2aBAB)(baB^2A) = b^2aB^3A$:
\begin{equation}
b^2aB^3A \in X
\end{equation}
Similarly the product of (2) and (5) gives $(B^2Abab)(BAb^2a) = B^2Ab^3a$
\begin{equation}
B^2Ab^3a \in X
\end{equation}
Finally, conjugate (9) by $A^4$ and multiply by (8) to conclude
\begin{align*}
a^4(B^2Ab^3a)A^4(b^2aB^3A) &= a^4B^2(Ab^3A)A^2b^2aB^3A \\
&= a^4B^2(b^2A^2b^2)A^2b^2aB^3A \\
&= a^2(b^2A^2b^2)aB^3A \\
&= a^2(Ab^3A)aB^3A \\
&= 1 \in X
\end{align*}\end{proof}
 
\noindent{\it Second proof.} We are interested in showing that $c:=[a,b^{-1}]$ is a generalized torsion element in the $5_2$ knot group. 
Obviously, if we're to have a hope, we'd better look at products of conjugates of $c$ by elements of the form $a^m b^n$ where $|m|,|n|$ are small. 
Note that $c^{a^{-1}}=[b^{-1},a^{-1}]$ and $c^b=[b,a]$. 
Since $a^{-2}b^2a=b^{-2}a^{-1}b^3$ (one form of the defining relation), we obtain
$c^{a^{-1}}c c^{b^{-1}}=bab^{-3}a^{-1}b$. 
Hence $(c^{a^{-1}}c c^{b^{-1}})^{b^{-1}a^2}=b^{-2}$. 
Now $c^{b^2}c^b=[b^2,a]$, so $c^{b^2}(c^{b^2}c^b)^{a^{-1}b}=b^{-2}a^{-1}ba^2b^{-2}a^{-1}b^3=b^{-2}a^{-1}b^3a$ (using $a^{-1}b^3=b^2a^{-2}b^2a$). 
So we can conclude that
$(c^{b^2}(c^{b^2}c^b)^{a^{-1}b})^{a^{-2}}=b^2$. 
Thus one product of conjugates of $c$ is the inverse of another product of conjugates of $c$, so $c$ is a generalized torsion element. 
 \qed

\section{Satellites and sums}

Satellites of knots are constructed as follows.  Let $K$ be a knot in the interior of a solid torus $V$ which in turn is standardly embedded in $\R^3$; that is, $V$ is a regular neighbourhood of the trivial knot.  We assume that $K$ is essential in $V$, in the sense that the inclusion induces an {\em injective} homomorphism 
$\pi_1(\partial V) \to \pi_1(V \setminus K).$  Now let $K_1$ be some other knot in $\R^3$ with tubular neighbourhood $N(K_1)$.  Since $N(K_1)$ is a solid torus, there is a homeomorphism $h:V \to N(K_1)$ -- in fact there are infinitely many isotopy classes of such homeomorphisms.  Finally, we let $K_2 = h(K)$; it is itself a knot in $\R^3$.  In this situation, we say that $K_2$ is a {\em satellite} of $K_1$ with {\em pattern} knot $K$.   By a van Kampen argument, one can see that the group of $K_1$ is isomorphic with a subgroup of the group of $K_2$.  This implies the following.

\begin{proposition}
If  $K_2$ is a {\em satellite} of $K_1$ and the knot group of $K_1$ has generalized torsion, then the same is true of the group of $K_2$.
\end{proposition}

\begin{corollary}
If one of the knots in a connected sum of knots has generalized torsion in its group, then the same is true of the sum.
\end{corollary}

This follows since a connected sum of two knots can be viewed as a satellite of either of the summands.  

\begin{corollary}  
Algebraic knots in the sense of Milnor \cite{Mil} have generalized torsion in their groups.
\end{corollary}

This is because they are iterated cables of the unknot, and therefore a satellite of a torus knot.

\begin{corollary}
There exist knots with trivial Alexander polynomial whose group contains generalized torsion.
\end{corollary}

For example, consider Whitehead doubles of knots with generalized torsion (see \cite{KL} p. 167, or \cite{Wh}).  If they are ``untwisted'' they have trivial polynomial, but being satellites they inherit generalized torsion in their groups.

\noindent \includegraphics[scale=0.5]{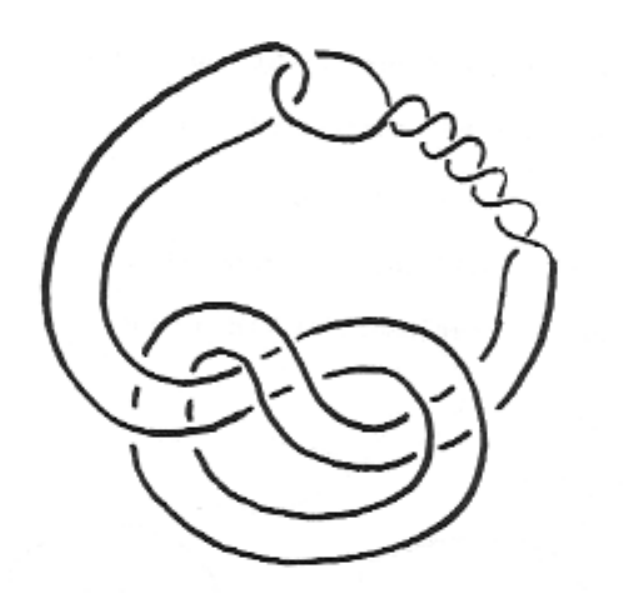}
Untwisted double of the trefoil.  It has trivial Alexander polynomial, but there is generalized torsion in its group.

The following shows that the Alexander polynomial cannot detect the nonexistence of generalized torsion. 

\begin{corollary}
For any knot $K$, there exists another knot with the same Alexander polynomial as $K$, but whose group contains generalized torsion.
\end{corollary}

\begin{proof}  Just take the connected sum of $K$ with, say, the doubled trefoil above.
\end{proof}

It is known that for nontrivial knots which are fibred \cite{CR}, or which have one-relator presentations of a particular form \cite{CGW}, if the Alexander polynomial has no positive real roots then the knot group is not bi-orderable. This begs the following question.

\begin{question}
If $K$ is a knot whose Alexander polynomial is nontrivial and has no real positive roots, does it follow that the group of $K$ contains generalized torsion?
\end{question}

\section{conclusion}

In summary, we have seen that all nontrivial torus knots, as well as the hyperbolic knot $5_2$, have generalized torsion in their groups.  Moreover, possessing generalized torsion in the group is preserved under taking satellites.

In all the cases we have discussed, the generalized torsion element identified in the knot group is a commutator in appropriate generators.  Since for any knot group $G$ the abelianization is infinite cyclic, any generalized torsion element must be in the kernel of the abelianization map $G \to \Z$, in other words, the commutator subgroup.  That it is a {\em single} commutator in all these cases is interesting.  

It is also interesting to compare the two knot groups, for $4_1$ and $5_2$.  The latter can be rewritten, exchanging $a$ with its inverse, so that we have 
$$ G_{4_1} \cong \langle a, b | ab^3a = ba^2b \rangle   \quad 
G_{5_2} \cong \langle a, b | ab^3a = b^2a^2b^2 \rangle.$$
Curiously, the latter has generalized torsion while $G_{4_1}$ does not.

From \cite{PR} we have infinitely many knots whose groups are bi-orderable, namely fibred knots whose Alexander polynomial has all roots real and positive.   These knot groups cannot contain generalized torsion by Proposition \ref{bi-ordNoGenTors}. Although this is an infinite class, it does not seem to be proportionately large.  There are certainly many (most) knots for which bi-orderability of its group is an open question.  The same is true of generalized torsion.


\begin{thebibliography}{99}

\bibitem{BRW} S. Boyer, D. Rolfsen and B. Wiest, \emph{Orderable 3-manifold groups}, Ann. Institut Fourier (Grenoble) {\bf 55} (2005).

\bibitem{CGW}
I. M. Chiswell, A. M. W. Glass and J. S. Wilson, \emph{Residual nilpotence and ordering in one-relator groups and knot groups}, preprint,  arXiv:1405.0994

\bibitem{CDN} A. Clay, C. Desmarais, and P. Naylor, \emph{Testing bi-orderability of knot groups.}  In preparation.

\bibitem{CR}
A. Clay and D. Rolfsen, \emph{Ordered groups, eigenvalues, knots, surgery and $L$-spaces.} Math. Proc. Cambridge Philos. Soc. {\bf 152} (2012), no. 1, 115--129. 

\bibitem{HS} J. Howie and H. Short, \emph{The band-sum problem,} J. London Math. Soc. (2) {\bf 31} (1985), 571 -- 576.  

\bibitem{Mil} J. Milnor,  Singular points of complex hypersurfaces,  Princeton University Press, 1968.

\bibitem{MR} R. Botto Mura and A. Rhemtulla,  \emph{Orderable Groups}, Lecture Notes in Pure and Applied Mathematics, Vol. 27, Marcel Dekker, 1977.

\bibitem{OS}
P. Ozsv{\'a}th and Z. Szab{\'o}, \emph{On knot {F}loer homology and lens space surgeries.} Topology, 
{\bf 44} (2005), 1281--1300

\bibitem{PR}
B. Perron and D. Rolfsen, \emph{On orderability of fibred knot groups.}
Math. Proc. Cambridge Philos. Soc.  {\bf 135}  (2003),  no. 1, 147--153.

\bibitem{KL} D. Rolfsen, Knots and Links, AMS Chelsea, Providence, RI, 2003.

\bibitem{Wh} J. H. C. Whitehead, \emph{On doubled knots,} J. London Math. Soc. {\bf 12} (1937), 63--71.

\end{thebibliography}
\end{document}